\documentclass{article}
\usepackage{hyperref}
\usepackage{cite}
\usepackage{amsfonts}
\usepackage{amsthm}
\usepackage{amsmath}
\usepackage{amssymb}
\usepackage{mathtools} 

\usepackage{authblk} 

\usepackage{xcolor}
\usepackage{physics} 

\usepackage{epsfig}

\usepackage{soul} 

\usepackage{subcaption}
\usepackage{graphicx}
\usepackage{wrapfig}
\usepackage{tikz}
\usepackage{pgfplots}
\usepackage{pgfplotstable}
\usepackage{tkz-fct}
\usetikzlibrary{pgfplots.polar}
\pgfplotsset{compat=newest}
\usepgfplotslibrary{fillbetween}
\usetikzlibrary{patterns}
\usepackage[top=1in, bottom=1in, left=1in, right=1in]{geometry}

\usepackage{mathrsfs}
\usepackage{float}
\hypersetup{colorlinks, linkcolor=red, 
filecolor=darkgreen, urlcolor=blue, citecolor=blue}
\theoremstyle{plain}
\newtheorem{Theorem}{Theorem}[section]
\newtheorem{Lemma}{Lemma}[section]

\newtheorem{Corollary}{Corollary}[section]
\newtheorem{Remark}{Remark}[section]
\theoremstyle{Definition}
\newtheorem{Definition}{Definition}[section]

\numberwithin{equation}{section}

\definecolor{pomcol}{rgb}{1,0,0}
\definecolor{dbcol}{rgb}{0,0,0.8}

\makeatletter
\newcommand{\subjclass}[2][]{
  \let\@oldtitle\@title%
  \gdef\@title{\@oldtitle\footnotetext{#1 
	\emph{MSC $(2020):$} #2}}
}
\newcommand{\keywords}[1]{
  \let\@@oldtitle\@title%
  \gdef\@title{\@@oldtitle\footnotetext{\emph{Keywords:} #1}}
}
\makeatother

\begin{document}

\title{{Hankel Determinant for a General Subclass of $m$-Fold Symmetric Bi-Univalent Functions Defined by Ruscheweyh Operators}}

\date{}

\author[1]{Pishtiwan Othman Sabir} 
\author[2]{Ravi P. Agarwal} 
\author[1]{Shabaz Jalil Mohammed Faeq} 
\author[3]{Pshtiwan Othman Mohammed\thanks{Corresponding author. Email: \texttt{pshtiwansangawi@gmail.com}}}
\author[4]{Nejmeddine Chorfi} 
\author[5]{Thabet Abdeljawad\thanks{Corresponding author. Email: \texttt{tabdeljawad@psu.edu.sa}}}

\affil[1]{{\small Department of Mathematics, College of Science, University of Sulaimani, Sulaymaniyah 46001, Kurdistan Region, Iraq}}
\affil[2]{{\small Department of Mathematics, Texas A\& M University-Kingsville, Kingsville, TX, 78363, USA}}
\affil[3]{{\small Department of Mathematics, College of Education, University of Sulaimani, Sulaimani 46001, Kurdistan Region, Iraq}}
\affil[4]{{\small Department of Mathematics, College of Science, King Saud University, P.O. Box 2455, Riyadh 11451, Saudi Arabia}}
\affil[5]{{\small Department of Mathematics and Sciences, Prince Sultan University, P.O. Box 66833, 11586 Riyadh, Saudi Arabia}}

\keywords{Analytic and univalent functions; $m$-fold symmetric univalent functions; $m$-fold symmetric bi-univalent functions; Ruscheweyh derivative; Hankel determinant}
\subjclass{30C45; 30C50; 26A51; 26B05; 15A15}

\maketitle

\begin{abstract}
Making use of the Hankel determinant and the Ruscheweyh derivative, in this work, we consider a general subclass of $m$-fold symmetric normalized bi-univalent functions defined in the open unit disk. Moreover, we investigate the bounds for the second Hankel determinant of this class and some consequences of the results are presented. In addition, to demonstrate the accuracy on some functions and conditions, most general programs are written in Python V.3.8.8 (2021).
\end{abstract}

\section{Introduction}
Let $\mathcal{A}$ denote the class of the analytic functions $f$ in the open unit disk $\mathbb{U}=\{z \in \mathbb{C}:|z|<1\}$, normalized by the conditions $f(0)=f^{\prime}(0)-1=0$ of the Taylor-Maclaurian series expansion

\begin{equation}
    f(z)=z+\sum_{n=2}^{\infty} a_{k} z^{k}.
\end{equation}

Further, assume that $\mathcal{S}$ denotes the subclass of $\mathcal{A}$ which contains all univalent functions in $\mathbb{U}$ satisfying (1.1) and  $\mathcal{P}$ represents the subclass of all functions $h(z)$ of the form

\begin{equation}
    h(z)=1+h_{1} z+h_{2} z^{2}+h_{3} z^{3}+\cdots
\end{equation}
which are analytic in the open unit disk $\mathbb{U}$ and $\operatorname{Re}(h(z))>0, z \in \mathbb{U}$.

For a function $f \in \mathcal{A}$ defined by (1.1), the Ruscheweyh derivative operator (see \cite{A1}) is defined by

$$
\mathcal{R}^{\gamma} f(z)=z+\sum_{k=2}^{\infty} \Omega(\gamma, k) a_{k} z^{k},
$$
where $\delta \in \mathbb{N}_{0}=\{0,1,2, \ldots\}=\mathbb{N} \cup\{0\}, z \in \mathbb{U}$, and
$$
\Omega(\gamma, k)=\frac{\Gamma(\gamma+k)}{\Gamma(k) \Gamma(\gamma+1)}.
$$

The Koebe 1/4-theorem (see \cite{A2}) asserts that every univalent function $f \in \mathcal{S}$ has an inverse $f^{-1}$ defined by

$$
f^{-1}(f(z))=z \quad(z \in \mathbb{U}) \text { and } f\left(f^{-1}(w)\right)=w \quad\left(|w|<r_{0}(f), r_{0}(f) \geq \frac{1}{4}\right).
$$

The inverse function $g=f^{-1}$ has the form

\begin{equation}
    g(w)=f^{-1}(w)=w-a_{2} w^{2}+\left(2 a_{2}^{2}-a_{3}\right) w^{3}-\left(5 a_{2}^{3}-5 a_{2} a_{3}+a_{4}\right) w^{4}+\cdots.
\end{equation}

A function $f \in \mathcal{A}$ is said to be bi-univalent if both $f$ and $f^{-1}$ satisfy the univalent property. The class of bi-univalent functions in $\mathbb{U}$ is denoted by $\Sigma$. Some examples of functions in the class $\Sigma$ are given as follows:

$$
\frac{z}{1-z}, \quad-\log (1-z) \quad \text { and } \quad \frac{1}{2} \log \left(\frac{1+z}{1-z}\right),
$$
with the corresponding inverse functions

$$
\frac{e^{w}-1}{e^{w}}, \quad \frac{w}{1+w} \text { and } \quad \frac{e^{2 w}-1}{e^{2 w}+1},
$$
respectively.

Determination of the estimates for the Taylor-Maclaurin coefficients $a_{n}$ is a crucial problem in geometric function theory and provides knowledge about the geometric characteristics of these functions. Lewin \cite{A3} investigated the class $\Sigma$ of bi-univalent functions and showed that $\left|a_{2}\right|<1.51$ for the functions belonging to the class $\Sigma$.  Brannan and Clunie \cite{A4} conjectured that $\left|a_{2}\right| \leq \sqrt{2}$. Subsequently, Netanyahu \cite{A5} showed that max $\left|a_{2}\right|=\frac{4}{3}$ for $f \in$ $\Sigma$. Srivastava et al. \cite{A6} improved the investigation for various subclasses of the bi-univalent function class $\Sigma$ and established bounds on $\left|a_{2}\right|$ and $\left|a_{3}\right|$ in recent years. Many recent works are devoted to studying the bi-univalent functions class $\Sigma$ and obtaining non-sharp bounds on the Taylor-Maclaurin coefficients $\left|a_{2}\right|$ and $\left|a_{3}\right|$ (see, for example, \cite{A8,A7,A28, A29,A30}). However, coefficient estimates bound of $\left|a_{n}\right|(n \in\{4,5,6, \ldots\})$ for a function $f \in \Sigma$ defined by (1.1) remains an open problem. In fact, there is no natural way to get upper bound for coefficients greater than three. In exceptional cases, there are some articles in which  Faber polynomial techniques were used for finding upper bounds for higher order coefficients (see, for example, \cite{A11,A10,A31}).

The Hankel determinant is a valuable tool in studying univalent functions whose components are coefficients of functions in the subclasses of $\mathcal{S}$. The Hankel determinants $H_{q}(n)$ $(n, q \in \mathbb{N})$ of the function $f$ are defined by (see \cite{A12})

$$
H_{q}(n)=\left|\begin{array}{cccc}
a_{n} & a_{n+1} & \cdots & a_{n+q-1} \\
a_{n+1} & a_{n+2} & \cdots & a_{n+q} \\
\vdots & \vdots & & \vdots \\
a_{n+q-1} & a_{n+q} & \cdots & a_{n}+2 q-2
\end{array}\right| \quad\left(a_{1}=1\right).
$$

Note

$$
H_{2}(1)=\left|\begin{array}{ll}
a_{1} & a_{2} \\
a_{2} & a_{3}
\end{array}\right|,
$$
and

$$
H_{2}(2)=\left|\begin{array}{ll}
a_{2} & a_{3} \\
a_{3} & a_{4}
\end{array}\right|.
$$

Estimates for upper bounds of $\left|H_{2}(1)\right|=\left|a_{3}-a_{2}^{2}\right|$ and $\left|H_{2}(2)\right|=\left|a_{2} a_{4}-a_{3}^{2}\right|$ are called Fekete-Szegö and second Hankel determinant problems, respectively. Additionally, Fekete and Szegö \cite{A13} proposed the summarized functional $a_{3}-\mu a_{2}^{2}$, in which $\mu$ is some real number. Lee et al. \cite{A14} presented a concise overview of Hankel determinants for analytic univalent functions and obtained bounds for $\mathrm{H}_{2}(2)$  for functions belonging to some classes defined by subordination. The estimation of $|\mathrm{H}_{2}(2)|$  has been the focus of recent Hankel determinant papers (see, for example, \cite{A17,A18,A16,A15,A32}).

For each function $f \in \mathcal{S}$, the function

\begin{equation}
    h(z)=\left(f\left(z^{m}\right)\right)^{\frac{1}{m}}, \quad(z \in \mathbb{U}, m \in \mathbb{N})
\end{equation}
is univalent and maps the unit disk into a region with $m$-fold symmetry. A function $f$ is said to be $m$-fold symmetric (see \cite{A19}) and denoted by $\mathcal{A}_{m}$, if it has the following normalized form:

\begin{equation}
    f(z)=z+\sum_{k=1}^{\infty} a_{m k+1} z^{m k+1}, \quad(z \in \mathbb{U}, m \in \mathbb{N}).
\end{equation}

We denote by $\mathcal{S}_{m}$ the class of $m$-fold symmetric univalent functions in $\mathbb{U}$, which are normalized by the series expansion (1.5). In fact, the functions in the class $\mathcal{S}$ are 1-fold symmetric. In view of the work of Koepf \cite{A19} the $m$-fold symmetric function $h \in \mathcal{P}$ is of the form

\begin{equation}
    h(z)=1+h_{m} z^{m}+h_{2 m} z^{2 m}+h_{3 m} z^{3 m}+\cdots.
\end{equation}

Analogous to the concept of $m$-fold symmetric univalent functions, Srivastava et al. \cite{A20} defined the concept of $m$ fold symmetric bi-univalent functions in a direct way. Each function $f \in \Sigma$ generates an $m$-fold symmetric bi-univalent function for each $m \in \mathbb{N}$. The normalized form of $f$ is given as (1.5) and the extension $g=f^{-1}$ is  as follows:

\begin{equation}
    \begin{aligned}
& g(w)=w-a_{m+1} w^{m+1}+\left[(m+1) a_{m+1}^{2}-a_{2 m+1}\right] w^{2 m+1} \\
&- {\left[\frac{1}{2}(m+1)(3 m+2) a_{m+1}^{3}-(3 m+2) a_{m+1} a_{2 m+1}+a_{3 m+1}\right] w^{3 m+1}+\cdots }.
\end{aligned}
\end{equation}

We denote by $\Sigma_{\mathrm{m}}$ the class of $m$-fold symmetric bi-univalent functions in $\mathbb{U}$. For $m=1$, the series (1.7) coincides with the series (1.3) of the class $\Sigma$. Some examples of $m$-fold symmetric bi-univalent functions are given as follows:

$$
\left[\frac{z^{m}}{1-z^{m}}\right]^{\frac{1}{m}}, \quad\left[-\log \left(1-z^{m}\right)\right]^{\frac{1}{m}} \quad \text { and } \quad\left[\frac{1}{2} \log \left(\frac{1+z^{m}}{1-z^{m}}\right)\right]^{\frac{1}{m}},
$$
with the corresponding inverse functions

$$
\left(\frac{w^{m}}{1+w^{m}}\right)^{\frac{1}{m}}, \quad\left(\frac{e^{w^{m}}-1}{e^{w^{m}}}\right)^{\frac{1}{m}} \quad \text { and } \quad\left(\frac{e^{2 w^{m}}-1}{e^{2 w^{m}}+1}\right)^{\frac{1}{m}},
$$
respectively.

Recently, some authors studied the $m$-fold symmetric bi-univalent function class $\Sigma_{\mathrm{m}}$ (see, for example, \cite{A22,A23,A21,A33}) and obtained non-sharp bound estimates on the first two Taylor-Maclaurin coefficients $\left|a_{m+1}\right|$ and $\left|a_{2 m+1}\right|$. In this respect, Altinkaya and Yalçin \cite{A24} obtained non-sharp estimates on the second Hankel determinant for the subclass $H_{\Sigma_{m}}(\beta)$ of the $m$-fold symmetric bi-univalent function class $\Sigma_{\mathrm{m}}$.

For a function $f \in \mathcal{A}_{m}$ defined by (1.5), analogous to the Ruscheweyh derivative $\mathcal{R}^{\gamma}: \mathcal{A} \rightarrow \mathcal{A}$, the $m$-fold Ruscheweyh derivative $\mathcal{R}^{\gamma}: \mathcal{A}_{m} \rightarrow \mathcal{A}_{m}$ is defined as follows (see \cite{A9}):
$$
\mathcal{R}^{\gamma} f(z)=z+\sum_{k=1}^{\infty} \frac{\Gamma(\gamma+k+1)}{\Gamma(k+1) \Gamma(\gamma+1)} a_{m k+1} z^{m k+1} \quad\left(\gamma \in \mathbb{N}_{0}, m \in \mathbb{N}, z \in \mathbb{U}\right).
$$

Considering the significant role of the Hankel determinant in recent years, the object of this paper, is to study estimates for $\left|H_{2}(2)\right|$ of a general subclass of $m$-fold symmetric bi-univalent functions in $\mathbb{U}$ by applying the $m$-fold Ruscheweyh derivative operator and to  obtain upper bounds on $\left|a_{m+1} a_{3 m+1}-a_{2 m+1}^{2}\right|$ for functions in the subclass $\Xi_{\Sigma_{m}}(\lambda, \gamma ; \beta)$.

In order to derive our main results, we need to the following lemmas that will be useful in proving the basic theorem of section 2.

\begin{Lemma}\cite{A2}
If the function $h \in \mathcal{P}$ is given by the series (1.2), then 

\begin{equation}
    \left|h_{k}\right| \leq 2\quad(k \in \mathbb{N}),
\end{equation}
and

\begin{equation}
    \left|h_{2}-\frac{h_{1}^{2}}{2}\right| \leq 2-\frac{\left|h_{2}\right|^{2}}{2}.
\end{equation}
\end{Lemma}

\begin{Lemma}\cite{A25}
    If the function $h \in \mathcal{P}$ is given by the series (1.2), then

\begin{equation}
    2 h_{2}=h_{1}^{2}+x\left(4-h_{1}^{2}\right),
\end{equation}
and

\begin{equation}
    4 h_{3}=h_{1}^{3}+2\left(4-h_{1}^{2}\right) h_{1} x-h_{1}\left(4-h_{1}^{2}\right) x^{2}+2\left(4-h_{1}^{2}\right)\left(1-|x|^{2}\right) z,
\end{equation}
for some $x, z$ with $|x| \leq 1$ and $|z| \leq 1$.
\end{Lemma}
 

\section{The Main Result and Consequences}
Our main result in this section, is to study estimates for the second Hankel determinant of the subclass $\Xi_{\Sigma_{m}}(\lambda, \gamma ; \beta)$ of $m$-fold symmetric bi-univalent functions in $\mathbb{U}$, and we show that our results are an improvement on the existing coefficient estimates.
\begin{Definition}
 A function $f \in \Sigma_{m}$ given by (1.5) is said to be in the class 

 $\Xi_{\Sigma_{m}}(\lambda, \gamma; \beta) \; (\lambda \geq 1,
 $ $
 \gamma \in \mathbb{N}_{0}, 0 \leq \beta<1$ and $m \in$ $\mathbb{N}$) if it satisfies the conditions

\begin{equation}
    \operatorname{Re}\left\{(1-\lambda) \frac{\mathcal{R}^{\gamma} f(z)}{z}+\lambda\left(\mathcal{R}^{\gamma} f(z)\right)^{\prime}\right\}>\beta,
\end{equation}
and

\begin{equation}
    \operatorname{Re}\left\{(1-\lambda) \frac{\mathcal{R}^{\gamma} f(w)}{w}+\lambda\left(\mathcal{R}^{\gamma} f(w)\right)^{\prime}\right\}>\beta,
\end{equation}
where $z, w \in \mathbb{U}$ and the function $g=f^{-1}$ is given by (1.7).
\end{Definition}

\begin{Theorem}
Let $f \in \Xi_{\Sigma_{m}}(\lambda, \gamma ; \beta)$ be given by (1.5). Then

$$
\left|a_{m+1} a_{3 m+1}-a_{2 m+1}^{2}\right| \leq\left\{\begin{array}{c}

\frac{4(1-\beta)^{2}}{(\gamma+1)^{2}(m \lambda+1)}\left[\frac{(m+1)^{2}(1-\beta)^{2}}{(\gamma+1)^{2}(m \lambda+1)^{3}}+\frac{6}{(\gamma+2)(\gamma+3)(3 m \lambda+1)}\right], \quad\qquad \beta \in[0, \tau] \\
\\
\frac{4(1-\beta)^{2}}{(\gamma+1)^{2}(\gamma+2)^{2}(2 m \lambda+1)^{2}}\left[4-\frac{\left[\omega_{2}(1-\beta)+9 \omega_{3}-4 \omega_{4}\right]^{2}}{\omega_{4}\left[\omega_{1}(1-\beta)^{2}-2 \omega_{2}(1-\beta)-12 \omega_{3}+4 \omega_{4}\right]}\right], \beta \in[\tau, 1)
\end{array}\right.
$$
where

\begin{eqnarray}
&&\omega_{1}:=(m+1)^{2}(\gamma+2)^{2}(\gamma+3)(2 m \lambda+1)^{2}(3 m \lambda+1),\\&&\omega_{2}:=m(\gamma+1)(\gamma+2)(\gamma+3)(m \lambda+1)^{2}(2 m \lambda+1)(3 m \lambda+1),\\&&\omega_{3}:=(\gamma+1)^{2}(\gamma+2)(m \lambda+1)^{3}(2 m \lambda+1)^{2},\\&&\omega_{4}:=(\gamma+1)^{2}(\gamma+3)(m \lambda+1)^{4}(3 m \lambda+1),
\end{eqnarray}
and

$
\qquad\qquad\qquad\qquad\quad\tau:=1-\frac{\omega_{2}+\sqrt{\omega_{2}^{2}+12 \omega_{1} \omega_{3}}}{2 \omega_{1}}.
$
\end{Theorem}

\begin{proof}
It follows from (2.1) and (2.2) that there exist $p$ and $q$ in the class $\mathcal{P}$ such that

\begin{equation}
    (1-\lambda) \frac{\mathcal{R}^{\gamma} f(z)}{z}+\lambda\left(\mathcal{R}^{\gamma} f(z)\right)^{\prime}=\beta+(1-\beta) p(z),
\end{equation}
and

\begin{equation}
    (1-\lambda) \frac{\mathcal{R}^{\gamma} f(w)}{w}+\lambda\left(\mathcal{R}^{\gamma} f(w)\right)^{\prime}=\beta+(1-\beta) q(z)
\end{equation}
where $p$ and $q$ are given by the series (1.6).

We also find that

\begin{equation}
    \begin{aligned}
(1-\lambda)& \frac{\mathcal{R}^{\gamma} f(z)}{z}+ \lambda\left(\mathcal{R}^{\gamma} f(z)\right)^{\prime} \\
& =1+(\gamma+1)(m \lambda+1) a_{m+1} z^{m}+\frac{1}{2}(\gamma+1)(\gamma+2)(2 \lambda m+1) a_{2 m+1} z^{2 m} \\
& +\frac{1}{6}(\gamma+1)(\gamma+2)(\gamma+3)(3 \lambda m+1) a_{3 m+1} z^{3 m}+\cdots,
\end{aligned}
\end{equation}
and

\begin{equation}
    \begin{aligned}
&(1-\lambda) \frac{\mathcal{R}^{\gamma} g(w)}{w}  +\lambda\left(\mathcal{R}^{\gamma} g(w)\right)^{\prime} \\
& =1-(\gamma+1)(m \lambda+1) a_{m+1} w^{m}+\frac{1}{2}(\gamma+1)(\gamma+2)(2 m \lambda +1)\times\\
& \left[(m+1) a_{m+1}^{2}-a_{2 m+1}\right] w^{2 m}-\frac{1}{6}(\gamma+1)(\gamma+2)(\gamma+3)(3 m \lambda +1)\times\\
& \left[\frac{1}{2}(m+1)(3 m+2) a_{m+1}^{3}-(3 m+2) a_{m+1} a_{2 m+1}+a_{3 m+1}\right] w^{3 m}+\cdots.
\end{aligned}
\end{equation}

Equating coefficients in (2.7) and (2.8)  we have

\begin{eqnarray}
   &(\gamma+1)(m \lambda+1) a_{m+1}=(1-\beta) p_{m},
\end{eqnarray}

\begin{eqnarray}
	\frac{1}{2}(\gamma+1)(\gamma+2)(2 \lambda m+1) a_{2 m+1}=(1-\beta) p_{2 m}, 
\end{eqnarray}

\begin{eqnarray}
	\frac{1}{6}(\gamma+1)(\gamma+2)(\gamma+3)(3 m \lambda+1) a_{3 m+1}=(1-\beta) p_{3 m}, 
\end{eqnarray}
and

\begin{eqnarray}
    -(\gamma+1)(m \lambda+1) a_{m+1}=(1-\beta) q_{m}, 
\end{eqnarray}

\begin{eqnarray}
    \frac{1}{2}(\gamma+1)(\gamma+2)(2 m \lambda+1)\left[(m+1) a_{m+1}^{2}-a_{2 m+1}\right]=(1-\beta) q_{2 m},
\end{eqnarray}

\begin{equation}
\begin{aligned}
&-\frac{1}{6}(\gamma+1)(\gamma+2)(\gamma+3)(3 m \lambda+1)\times\\&\left[\frac{1}{2}(m+1)(3 m+2) a_{m+1}^{3}-(3 m+2) a_{m+1} a_{2 m+1}+a_{3 m+1}\right]=(1-\beta) q_{3 m}.
\end{aligned}
\end{equation}

From (2.11) and (2.14), we get

\begin{equation}
    p_{m}=-q_{m},
\end{equation}
and

\begin{equation}
    a_{m+1}=\frac{1-\beta}{(\gamma+1)(m \lambda+1)} p_{m}.
\end{equation}

Now, from (2.12) and (2.15) and (2.18), we obtain

\begin{equation}
    a_{2 m+1}=\frac{(m+1)(1-\beta)^{2}}{2(\gamma+1)^{2}(m \lambda+1)^{2}} p_{m}^{2}+\frac{(1-\beta)}{(\gamma+1)(\gamma+2)(2 m \lambda+1)}\left(p_{2 m}-q_{2 m}\right).
\end{equation}

Also, from (2.13), (2.16), (2.18) and (2.19), we find that

\begin{equation}
    \begin{gathered}
a_{3 m+1}=\frac{(3 m+2)(1-\beta)^{2}}{2(\gamma+1)^{2}(\gamma+2)(m \lambda+1)(2 m \lambda+1)} p_{m}\left(p_{2 m}-q_{2 m}\right)\\+\frac{3(1-\beta)}{(\gamma+1)(\gamma+2)(\gamma+3)(3 m \lambda+1)}\left(p_{3 m}\right. 
\left.-q_{3 m}\right).
\end{gathered}
\end{equation}

Then, from (2.18), (2.19) and (2.20) we have that

\begin{equation}
    \begin{aligned}
& a_{m+1} a_{3 m+1}-a_{2 m+1}^{2}=-\frac{(m+1)^{2}(1-\beta)^{4}}{4(\gamma+1)^{4}(m \lambda+1)^{4}} p_{m}^{4} \\
&+\frac{m(1-\beta)^{3}}{2(\gamma+1)^{3}(\gamma+2)(m \lambda+1)^{2}(2 m \lambda+1)} p_{m}^{2}\left(p_{2 m}-q_{2 m}\right) \\
&+\frac{3(1-\beta)^{2}}{(\gamma+1)^{2}(\gamma+2)(\gamma+3)(m \lambda+1)(3 m \lambda+1)} p_{m}\left(p_{3 m}-q_{3 m}\right) \\
&-\frac{(1-\beta)^{2}}{(\gamma+1)^{2}(\gamma+2)^{2}(2 m \lambda+1)^{2}}\left(p_{2 m}-q_{2 m}\right)^{2}.
\end{aligned}
\end{equation}

According to Lemma 1.2 and (2.17), we can write

\begin{equation}
    p_{2 m}-q_{2 m}=\frac{4-p_{m}^{2}}{2}(x-y),
\end{equation}
and

\begin{equation}
    \begin{aligned}
    p_{3 m}-q_{3 m}=&\frac{p_{m}^{3}}{2}+\frac{p_{m}\left(4-p_{m}^{2}\right)}{2}(x+y)-\frac{p_{m}\left(4-p_{m}^{2}\right)}{4}\left(x^{2}+y^{2}\right)\\&+\frac{4-p_{m}^{2}}{2}\left[\left(1-|x|^{2}\right) z-\left(1-|y|^{2}\right) w\right],
\end{aligned}
\end{equation}

\begin{equation}
    p_{2 m}+q_{2 m}=p_{m}^{2}+\frac{4-p_{m}^{2}}{2}(x+y),
\end{equation}
for some $x, y, z$ and $w$ with $|x| \leq 1,|y| \leq 1,|z| \leq 1$ and $|w| \leq 1$. Using (2.22) and (2.23) in (2.21)  we obtain 

$$
\begin{aligned}
& \left|a_{m+1} a_{3 m+1}-a_{2 m+1}^{2}\right| \\
& =\mid-\frac{(m+1)^{2}(1-\beta)^{4}}{4(\gamma+1)^{4}(m \lambda+1)^{4}} p_{m}^{4}+\frac{m(1-\beta)^{3}}{4(\gamma+1)^{3}(\gamma+2)(m \lambda+1)^{2}(2 m \lambda+1)} p_{m}^{2}\left(4-p_{m}^{2}\right)(x-y) \\
& +\frac{3(1-\beta)^{2}}{(\gamma+1)^{2}(\gamma+2)(\gamma+3)(m \lambda+1)(3 m \lambda+1)} p_{m}\left[\frac{p_{m}^{3}}{2}+\frac{p_{m}\left(4-p_{m}^{2}\right)}{2}(x+y)\right. \\
& -\frac{p_{m}\left(4-p_{m}^{2}\right)}{4}\left(x^{2}+y^{2}\right)+\frac{4-p_{m}^{2}}{2}\left[\left(1-|x|^{2}\right) z-\left(1-|y|^{2}\right) w\right] \\
& -\frac{(1-\beta)^{2}}{4(\gamma+1)^{2}(\gamma+2)^{2}(2 m \lambda+1)^{2}}\left(4-p_{m}^{2}\right)^{2}(x-y)^{2} \mid \\
\\ & \leq \frac{(m+1)^{2}(1-\beta)^{4}}{4(\gamma+1)^{4}(m \lambda+1)^{4}} p_{m}^{4}+\frac{3(1-\beta)^{2}}{2(\gamma+1)^{2}(\gamma+2)(\gamma+3)(m \lambda+1)(3 m \lambda+1)} p_{m}^{4} \\
& +\frac{3(1-\beta)^{2}}{(\gamma+1)^{2}(\gamma+2)(\gamma+3)(m \lambda+1)(3 m \lambda+1)} p_{m}\left(4-p_{m}^{2}\right) \\
& +\left[\frac{m(1-\beta)^{3}}{4(\gamma+1)^{3}(\gamma+2)(m \lambda+1)^{2}(2 m \lambda+1)} p_{m}^{2}\left(4-p_{m}^{2}\right)\right. \\
& \left.+\frac{3(1-\beta)^{2}}{2(\gamma+1)^{2}(\gamma+2)(\gamma+3)(m \lambda+1)(3 m \lambda+1)} p_{m}^{2}\left(4-p_{m}^{2}\right)\right](|x|+|y|) \\
& +\left[\frac{3(1-\beta)^{2}}{4(\gamma+1)^{2}(\gamma+2)(\gamma+3)(m \lambda+1)(3 m \lambda+1)} p_{m}^{2}\left(4-p_{m}^{2}\right)\right. \\
& \left.-\frac{3(1-\beta)^{2}}{2(\gamma+1)^{2}(\gamma+2)(\gamma+3)(m \lambda+1)(3 m \lambda+1)} p_{m}\left(4-p_{m}^{2}\right)\right]\left(|x|^{2}+|y|^{2}\right) \\
& +\frac{(1-\beta)^{2}}{4(\gamma+1)^{2}(\gamma+2)^{2}(2 m \lambda+1)^{2}}\left(4-p_{m}^{2}\right)^{2}(|x|+|y|)^{2}.
\end{aligned}
$$

Since $p$ in the class $\mathcal{P}$, we have (Lemma 1.1) $\left|p_{m}\right| \leq 2$. Letting $p_{m}=\rho$, we may assume without loss of generality that $\rho \in[0,2]$. Thus, for $\mu_{1}=|x| \leq 1$ and $\mu_{2}=|y| \leq 1$, we get

$$
\left|a_{m+1} a_{3 m+1}-a_{2 m+1}^{2}\right| \leq F_{1}+F_{2}\left(\mu_{1}+\mu_{2}\right)+F_{3}\left(\mu_{1}^{2}+\mu_{2}^{2}\right)+F_{4}\left(\mu_{1}+\mu_{2}\right)^{2},
$$
where

\begin{eqnarray*}
   &&F_{1}=F_{1}(\rho)=\frac{(m+1)^{2}(1-\beta)^{4} \rho^{4}}{4(\gamma+1)^{4}(m \lambda+1)^{4}}+\frac{3(1-\beta)^{2} \rho^{4}}{2(\gamma+1)^{2}(\gamma+2)(\gamma+3)(m \lambda+1)(3 m \lambda+1)} \\&&\qquad\qquad\qquad
\quad+\frac{3(1-\beta)^{2} \rho\left(4-\rho^{2}\right)}{(\gamma+1)^{2}(\gamma+2)(\gamma+3)(m \lambda+1)(3 m \lambda+1)} \geq 0, \\&&
F_{2}=F_{2}(\rho)=\frac{m(1-\beta)^{3} \rho^{2}\left(4-\rho^{2}\right)}{4(\gamma+1)^{3}(\gamma+2)(m \lambda+1)^{2}(2 m \lambda+1)}\\&&\qquad\qquad\qquad+\frac{3(1-\beta)^{2} \rho^{2}\left(4-\rho^{2}\right)}{2(\gamma+1)^{2}(\gamma+2)(\gamma+3)(m \lambda+1)(3 m \lambda+1)} \geq 0, \\&&
F_{3}=F_{3}(\rho)=\frac{3(1-\beta)^{2} \rho^{2}\left(4-\rho^{2}\right)}{4(\gamma+1)^{2}(\gamma+2)(\gamma+3)(m \lambda+1)(3 m \lambda+1)}\\&&\qquad\qquad\qquad-\frac{3(1-\beta)^{2} \rho\left(4-\rho^{2}\right)}{2(\gamma+1)^{2}(\gamma+2)(\gamma+3)(m \lambda+1)(3 m \lambda+1)}
\leq 0, \\&& F_{4}=F_{4}(\rho)=\frac{(1-\beta)^{2}\left(4-\rho^{2}\right)^{2}}{4(\gamma+1)^{2}(\gamma+2)^{2}(2 m \lambda+1)^{2}} \geq 0. 
\end{eqnarray*}

\begin{figure}[H]
	\centering
	\begin{subfigure}[b]{0.45\linewidth}
	\includegraphics[width=\linewidth]{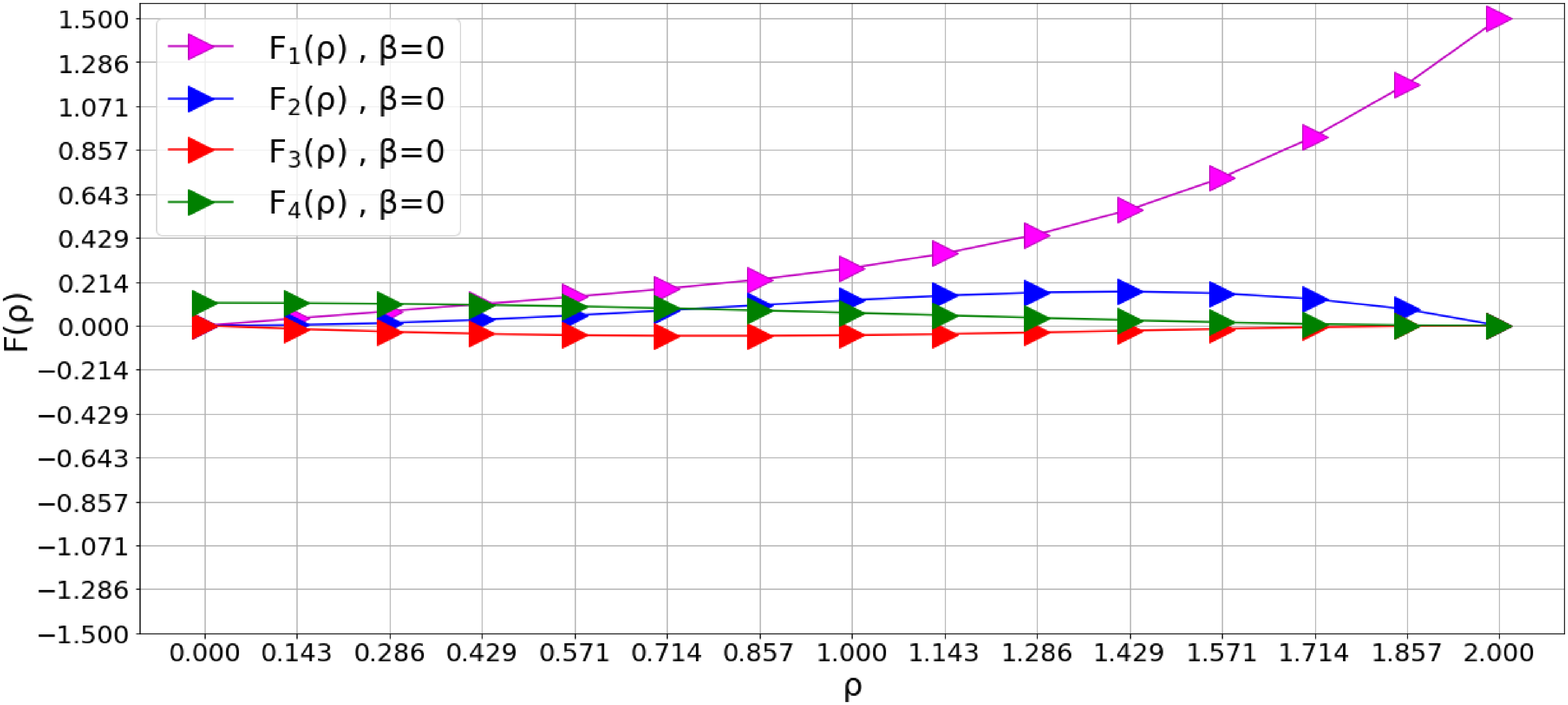}	
	\caption{$F_{1}, F_{2}, F_{3}$ and $F_{4}$ for $\beta=0$.}
	\end{subfigure}
	\begin{subfigure}[b]{0.45\linewidth}
	\includegraphics[width=\linewidth]{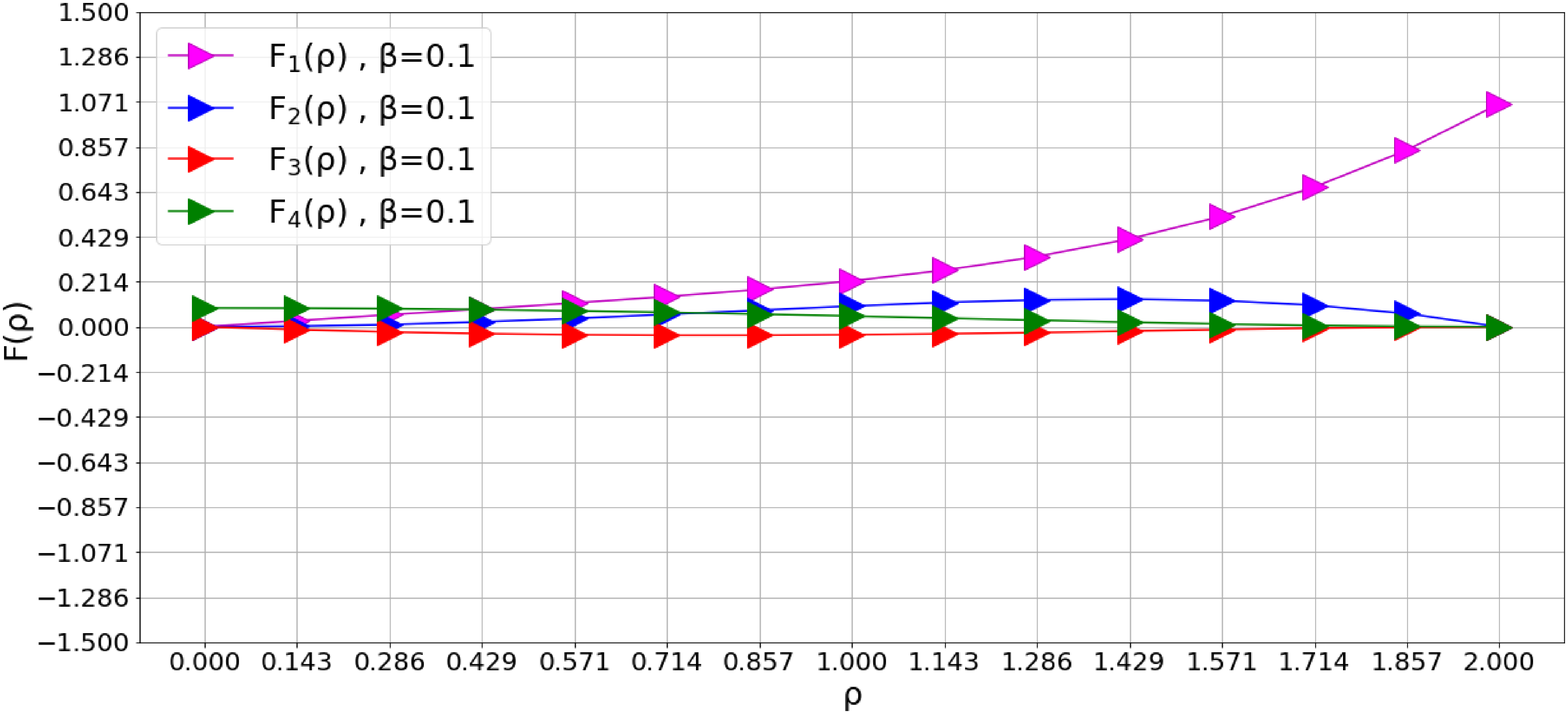}	
	\caption{$F_{1}, F_{2}, F_{3}$ and $F_{4}$ for $\beta=0.1$.}
	\end{subfigure}
    \begin{subfigure}[b]{0.45\linewidth}
	\includegraphics[width=\linewidth]{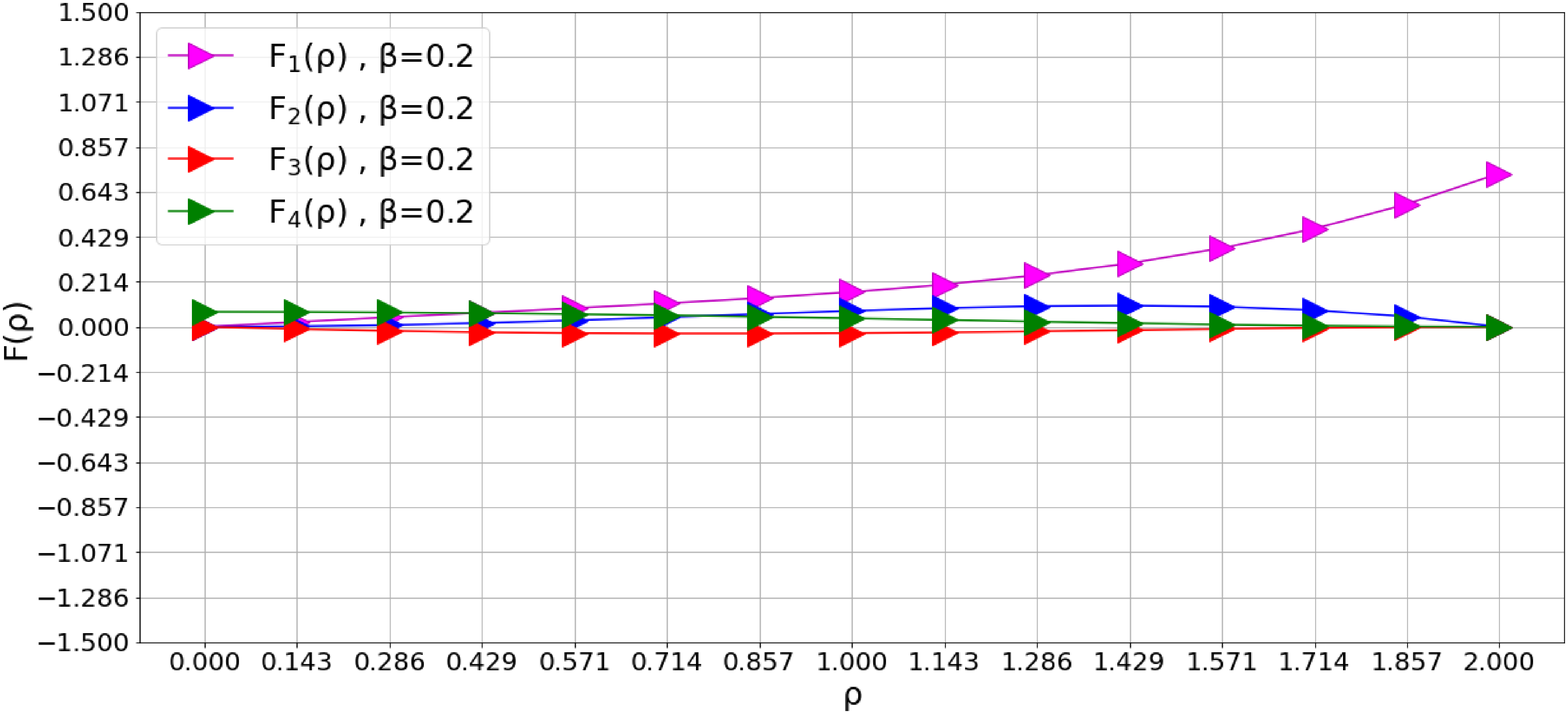}	
	\caption{$F_{1}, F_{2}, F_{3}$ and $F_{4}$ for $\beta=0.2$.}
    \end{subfigure}
    \begin{subfigure}[b]{0.45\linewidth}
    \includegraphics[width=\linewidth]{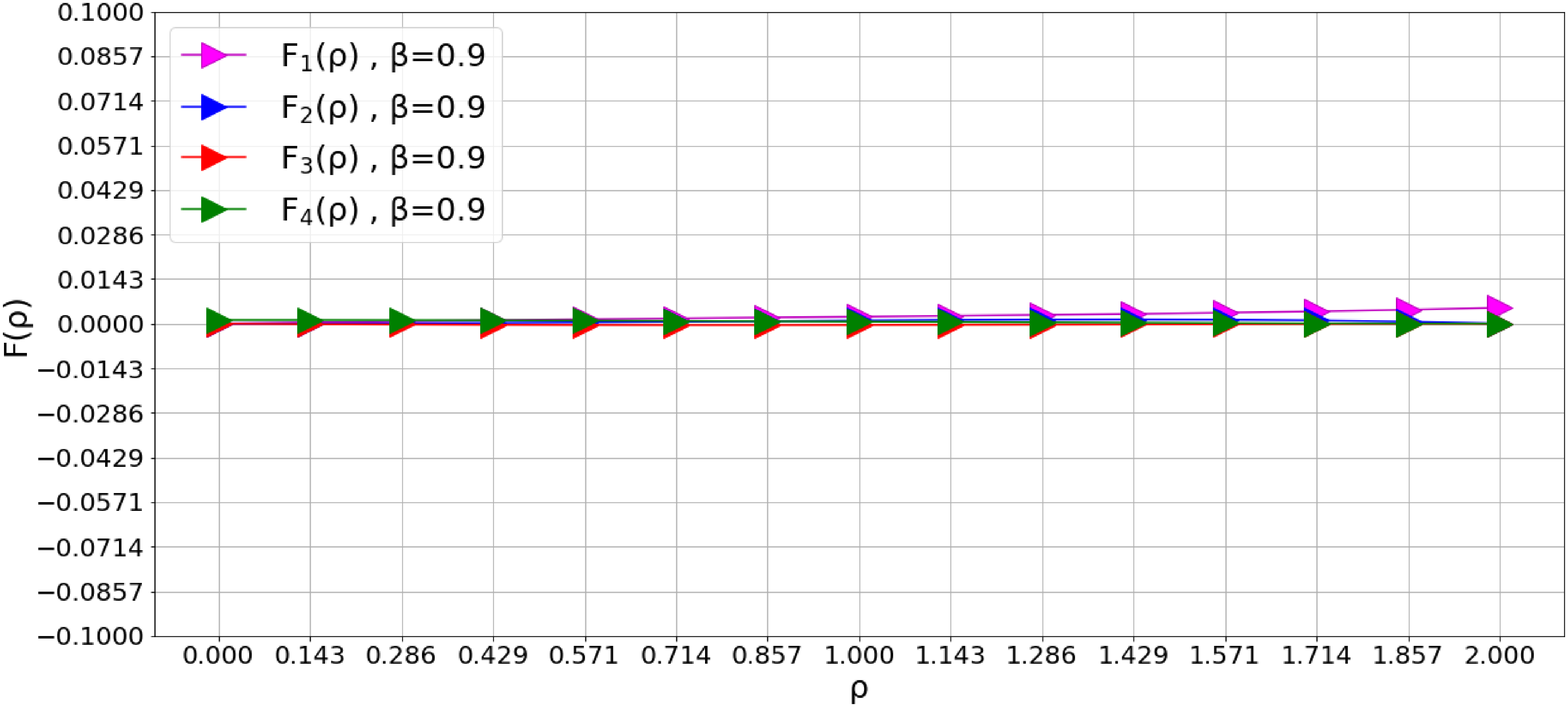}	
    \caption{$F_{1}, F_{2}, F_{3}$ and $F_{4}$ for $\beta=0.9$.}
    \end{subfigure}
\caption{Graph of $F_{1}, F_{2}, F_{3}$ and $F_{4}$  for $\gamma=0$ and $\lambda=m=1$.}
\end{figure}

Now, we need to maximize

$$F\left(\mu_{1}, \mu_{2}\right)=F_{1}+F_{2}\left(\mu_{1}+\mu_{2}\right)+F_{3}\left(\mu_{1}^{2}+\mu_{2}^{2}\right)+F_{4}\left(\mu_{1}+\mu_{2}\right)^{2}$$
in the closed square $\mathbb{S}=[0,1] \times[0,1]$ for $\rho \in[0,2]$. We investigate the maximum of $F\left(\mu_{1}, \mu_{2}\right)$ when $\rho \in(0,2), \rho=0$ and $\rho=2$, keeping in mind the sign of $$F_{\mu_{1} \mu_{1}} F_{\mu_{2} \mu_{2}}-\left(F_{\mu_{1} \mu_{2}}\right)^{2}$$ (according to the Second Derivative Test for functions of the two dependent variables $\mu_{1}$ and $\mu_{2}$)

First, let $\rho \in(0,2)$. Since $F_{3}<0$ and $F_{3}+2 F_{4}>0$ for $\rho \in(0,2)$, we see that 

$$
F_{\mu_{1} \mu_{1}} F_{\mu_{2} \mu_{2}}-\left(F_{\mu_{1} \mu_{2}}\right)^{2}= 4F_3\,(F_3+2 F_4)<0.
$$

\begin{figure}[H]
	\centering
	\includegraphics[width=0.8\textwidth]{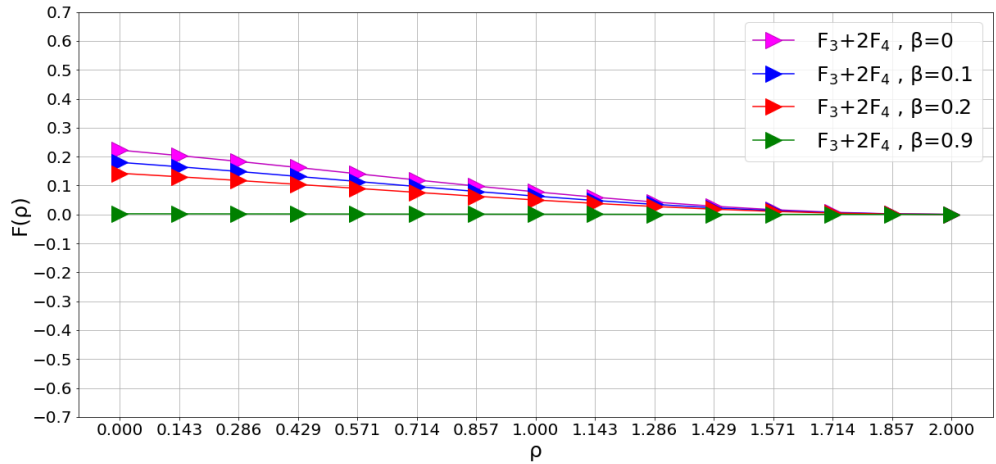}
	\caption{Graph of $F_{3}+2F_{4}$ for $\gamma=0$ and $\lambda=m=1$.}
\end{figure}

Thus, the function $F$ cannot have a local maximum in the interior of the square $\mathbb{S}$. Now, we investigate the maximum of $F$ on the boundary of the square $\mathbb{S}$.

Case 1. For $\mu_{1}=0$ and $\mu_{2} \in[0,1]$ (a similar argument can be applied for $\mu_{2}=0$ and $\mu_{1} \in[0,1]$, so we omit the details in that case), we obtain

$$
F\left(0, \mu_{2}\right)=G\left(\mu_{2}\right) \equiv F_{1}+F_{2} \mu_{2}+\left(F_{3}+F_{4}\right) \mu_{2}^{2}.
$$

Subcase 1. Let $F_{3}+F_{4} \geq 0$. In this case, for $0<\mu_{2}<1$ we have that

$$
G^{\prime}\left(\mu_{2}\right)=F_{2}+2\left(F_{3}+F_{4}\right) \mu_{2}>0,
$$
that is, $G\left(\mu_{2}\right)$ is an increasing function. Hence the maximum of $G\left(\mu_{2}\right)$ occurs at $\mu_{2}=1$ and
$$
\max \left\{F\left(0, \mu_{2}\right):\,\mu_2 \in [0,1] \right\}=\max \left\{G\left(\mu_{2}\right):\,\mu_2 \in [0,1] \right\}=G(1)=F_{1}+F_{2}+F_{3}+F_{4}.
$$

Subcase 2. Let $F_{3}+F_{4}<0$. Note
$$
F_{2}+2\left(F_{3}+F_{4}\right) \geq 0.$$
For $\mu_2 \in (0,1)$   since    $F_{3}+F_{4}<0$ we have   that

$$
F_{2}+2\left(F_{3}+F_{4}\right) \mu_{2} >F_{2}+2\left(F_{3}+F_{4}\right) \geq 0,
$$
so $G^{\prime}\left(\mu_{2}\right)>0$. Thus $\max \left\{G\left(\mu_{2}\right):\,\,\mu_2 \in [0,1]\right\}=G(1)$.

\begin{figure}[H]
	\centering
	\includegraphics[width=0.8\textwidth]{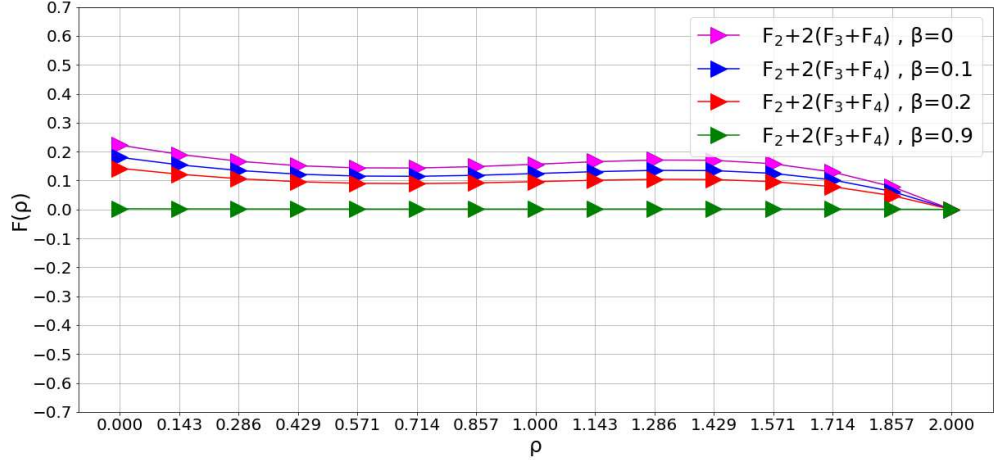}
	\caption{Graph of $F_{2}+2(F_{3}+F_{4})$ for $\gamma=0$ and $\lambda=m=1$.}
\end{figure}

Case 2. For $\mu_{1}=1$ and $\mu_{2}\in[0,1]$ (a similar argument can be applied for $\mu_{2}=1$ and $\mu_{1}\in[0,1]$, so we omit the details in that case), we obtain

$$
F\left(1, \mu_{2}\right)=H\left(\mu_{2}\right)=F_{1}+F_{2}+F_{3}+F_{4}+\left(F_{2}+2 F_{4}\right) \mu_{2}+\left(F_{2}+F_{4}\right) \mu_{2}^{2}.
$$

Thus an argument like in Subcases 1 and 2 yields

$$
\max \left\{F\left(1, \mu_{2}\right):\,\,\mu_2 \in [0,1] \right\}=\max \left\{H\left(\mu_{2}\right):\,\,\mu_2 \in [0,1]\right\}=H(1)=F_{1}+2\left(F_{2}+F_{3}\right)+4 F_{4}.
$$

\vspace{1mm}
Next let $\rho=2$.  Now let $(\mu_{1},\mu_{2})\in\mathbb{S}$ and note 

\begin{equation}
	F\left(\mu_{1}, \mu_{2}\right)=\frac{4(\gamma+2)(\gamma+3)(3 m \lambda+1)(m+1)^{2}(1-\beta)^{4}+24(\gamma+1)^{2}(m \lambda+1)^{3}(1-\beta)^{2}}{(\gamma+1)^{4}(\gamma+2)(\gamma+3)(m \lambda+1)^{4}(3 m \lambda+1)}.
\end{equation}

Keeping in mind the constant value in (2.25) we have

$$
\max \left\{F\left(\mu_{1}, \mu_{2}\right):\,\,\mu_1 \in [0,1], \,\mu_2 \in [0,1]\right\}=F(1,1)=F_{1}+2\left(F_{2}+F_{3}\right)+4 F_{4}.
$$

\vspace{1mm}
Finally, let $\rho=0$. Now let $(\mu_{1},\mu_{2})\in\mathbb{S}$ and note
$$F(\mu_{1},\mu_{2})=\frac{4(1-\beta)^{2}(\mu_{1}+\mu_{2})^2}{(\gamma+1)^{2}(\gamma+2)^{2}(2 m \lambda+1)^{2}}. $$

We see that, the maximum of $F\left(\mu_{1}, \mu_{2}\right)$ occurs at $\mu_{1}=\mu_{2}=1$ and
$$
\max \left\{F\left(\mu_{1}, \mu_{2}\right):\,\,\mu_1 \in [0,1], \,\mu_2 \in [0,1]\right\}=F(1,1)=F_{1}+2\left(F_{2}+F_{3}\right)+4 F_{4}.
$$

Combining all cases note since $F_{1}+2\left(F_{2}+F_{3}\right)+4 F_{4}\geq 0$ when $\rho \in [0,2]$ we have   
$$
\max \left\{F\left(\mu_{1}, \mu_{2}\right):\,\,\mu_1 \in [0,1], \,\mu_2 \in [0,1] \right\}=F(1,1).
$$

\begin{figure}[H]
	\centering
	\includegraphics[width=0.8\textwidth]{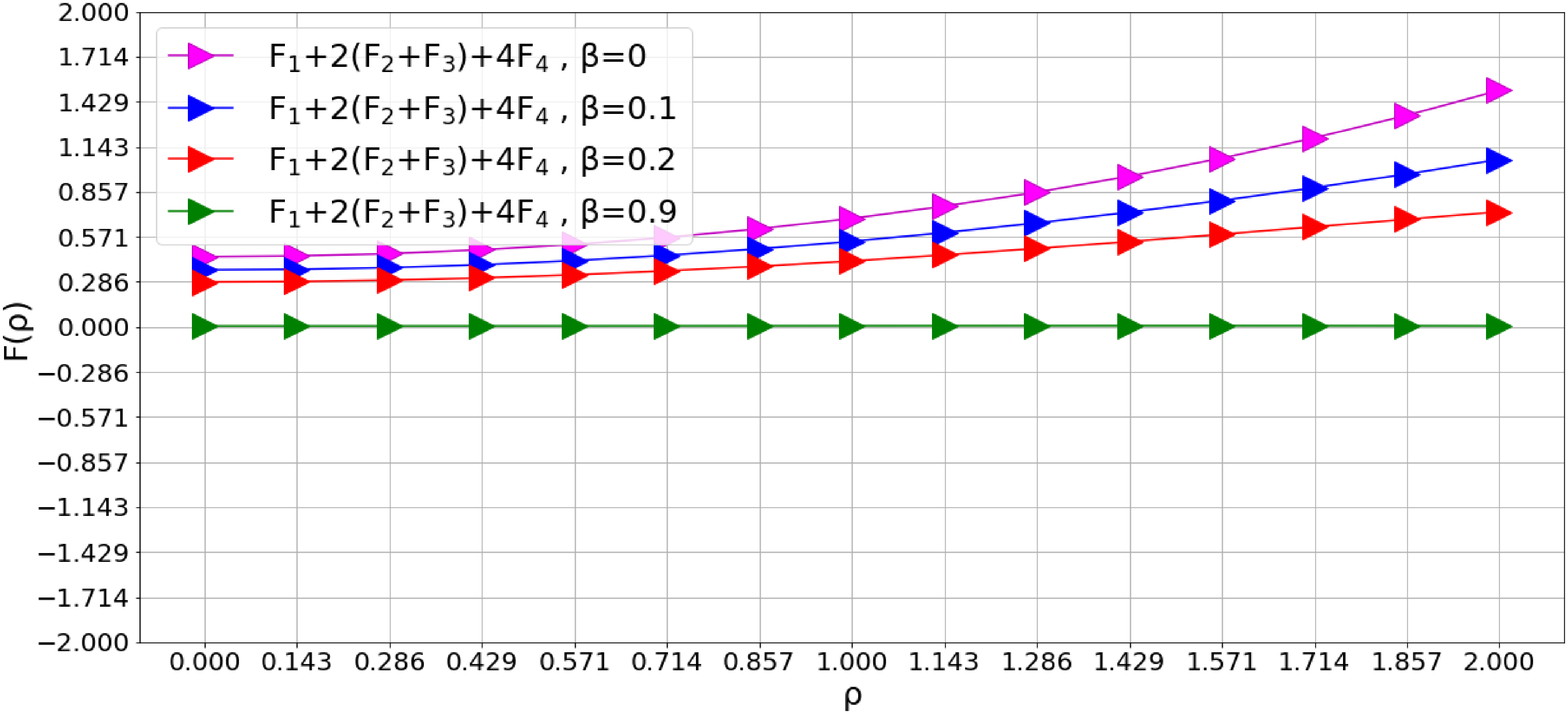}
	\caption{Graph of $F_{1}+2\left(F_{2}+F_{3}\right)+4 F_{4}$ for $\gamma=0$ and $\lambda=m=1$.}
\end{figure}

Let $K:[0,2] \rightarrow \mathbb{R}$ be given by

\begin{equation}
    K(\rho)=F(1,1)=F_{1}+2\left(F_{2}+F_{3}\right)+4 F_{4}.
\end{equation}

Substituting the values of $F_{1}, F_{2}, F_{3}$ and $F_{4}$ in the function $K$ defined by (2.26), yields

$$
\begin{aligned}
& K(\rho)=\frac{(1-\beta)^{2}}{4(\gamma+1)^{4}(\gamma+2)^{2}(\gamma+3)(1+m \lambda)^{4}(1+2 m \lambda)^{2}(1+3 m \lambda)}\times\\&\left[\left[(m+1)^{2}(\gamma+2)^{2}(\gamma+3)(2 m \lambda+1)^{2}(3 m \lambda+1)(1-\beta)^{2}\right.\right. \\&-2 m(\gamma+1)(\gamma+2)(\gamma+3)(m \lambda+1)^{2}(2 m \lambda+1)(3 m \lambda+1)(1-\beta) \\
&\left.-12(\gamma+1)^{2}(\gamma+2)(m \lambda+1)^{3}(2 m \lambda+1)^{2}+4(\gamma+1)^{2}(\gamma+3)(m \lambda+1)^{4}(3 m \lambda+1)\right] \rho^{4} \\
&+\left[8 m(\gamma+1)(\gamma+2)(\gamma+3)(1+m \lambda)^{2}(1+2 m \lambda)(1+3 m \lambda)(1-\beta)\right. \\
&\left.+72(\gamma+1)^{2}(\gamma+2)(m \lambda+1)^{3}(2 m \lambda+1)^{2}-32(\gamma+1)^{2}(\gamma+3)(m \lambda+1)^{4}(3 m \lambda+1)\right] \rho^{2} \\
&\left.+64(\gamma+1)^{2}(\gamma+3)(1+m \lambda)^{4}(1+3 m \lambda)\right].
\end{aligned}
$$
Now the maximum of $K(\rho)$ occurs either at $\rho=0$, $\rho \in (0,2)$ or $\rho=2$. Suppose first the maximum of $K(\rho)$ occurs at some $\rho \in (0,2)$. Note for any
$\rho \in(0,2)$  we have

$$
\begin{aligned}
& K^{\prime}(\rho)=\frac{(1-\beta)^{2}}{(\gamma+1)^{4}(\gamma+2)^{2}(\gamma+3)(1+m \lambda)^{4}(1+2 m \lambda)^{2}(1+3 m \lambda)}\times\\&\left[\left[(m+1)^{2}(\gamma+2)^{2}(\gamma+3)(2 m \lambda+1)^{2}(3 m \lambda+1)(1-\beta)^{2}\right.\right. \\&-2 m(\gamma+1)(\gamma+2)(\gamma+3)(m \lambda+1)^{2}(2 m \lambda+1)(3 m \lambda+1)(1-\beta) \\
&\left.-12(\gamma+1)^{2}(\gamma+2)(m \lambda+1)^{3}(2 m \lambda+1)^{2}+4(\gamma+1)^{2}(\gamma+3)(m \lambda+1)^{4}(3 m \lambda+1)\right] \rho^{3} \\
&+\left[4 m(\gamma+1)(\gamma+2)(\gamma+3)(\lambda m+1)^{2}(1+2 m \lambda)(1+3 m \lambda)(1-\beta)\right. \\
&\left.\left.+36(\gamma+1)^{2}(\gamma+2)(m \lambda+1)^{3}(2 m \lambda+1)^{2}-16(\gamma+1)^{2}(\gamma+3)(m \lambda+1)^{4}(3 m \lambda+1)\right] \rho\right].
\end{aligned}
$$

Next, we conclude the following results

Result 1. Let

$$
\omega_{1}(1-\beta)^{2}-2 \omega_{2}(1-\beta)-12 \omega_{3}+4 \omega_{4} \geq 0,
$$
that is,

$$
\beta \in\left[0,1-\frac{\omega_{2}+\sqrt{\omega_{2}^{2}+\omega_{1}\left[12 \omega_{3}-4 \omega_{4}\right]}}{\omega_{1}}\right],
$$
where $\omega_{1}, \omega_{2}, \omega_{3}$ and $\omega_{4}$ are given by (2.3), (2.4), (2.5) and (2.6), respectively.

Note, $K^{\prime}(\rho)>0$ for every $\rho \in(0,2)$. Thus,

\begin{equation*}
    \begin{aligned}
        \max \{K(\rho): 0< \rho < 2\}&=K(2^-)=\frac{4(1-\beta)^{2}}{(\gamma+1)^{2}(m \lambda+1)}\left[\frac{(m+1)^{2}(1-\beta)^{2}}{(\gamma+1)^{2}(m \lambda+1)^{3}}\right.\\&\qquad\qquad\qquad\qquad\qquad\qquad\left.+\frac{6}{(\gamma+2)(\gamma+3)(3 m \lambda+1)}\right].
    \end{aligned}
\end{equation*}

Result 2. Let

$$
\omega_{1}(1-\beta)^{2}-2 \omega_{2}(1-\beta)-12 \omega_{3}+4 \omega_{4}<0,
$$
that is,

$$
\beta \in\left(1-\frac{\omega_{2}+\sqrt{\omega_{2}^{2}+\omega_{1}\left[12 \omega_{3}-4 \omega_{4}\right]}}{\omega_{1}}, 1\right).
$$

Then $K^{\prime}(\rho)=0$ gives the  critical point $\rho_{1}=0$ or

$$
\rho_{2}=\sqrt{\frac{16 \omega_{4}-4 \omega_{2}(1-\beta)-36 \omega_{3}}{\omega_{1}(1-\beta)^{2}-2 \omega_{2}(1-\beta)-12 \omega_{3}+4 \omega_{4}}}.
$$

When

$$
\beta \in\left(1-\frac{\omega_{2}+\sqrt{\omega_{2}^{2}+\omega_{1}\left[12 \omega_{3}-4 \omega_{4}\right]}}{\omega_{1}}, 1-\frac{\omega_{2}+\sqrt{\omega_{2}^{2}+12 \omega_{1} \omega_{3}}}{2 \omega_{1}}\right],
$$
we observe that $\rho_{2} \geq 2$. Then  the maximum value of $K(\rho)$ occurs at $0^+$ or $2^-$.  This is a contradiction since we assumed the maximum of $K(\rho)$ occurs at some $\rho \in (0,2)$.

When

$$
\beta \in\left(1-\frac{\omega_{2}+\sqrt{\omega_{2}^{2}+12 \omega_{1} \omega_{3}}}{2 \omega_{1}}, 1\right),
$$
we observe that $\rho_{2} \in(0,2)$. Since $K^{\prime \prime}\left(\rho_{2}\right)<0$, the maximum value of $K(\rho)$ occurs at $\rho=\rho_{2}$. Thus, we have

\begin{equation}
	\begin{split}
 \max \{K(\rho): \rho \in (0,2)\}=K\left(\rho_{2}\right) &=\frac{4(1-\beta)^{2}}{(\gamma+1)^{2}(\gamma+2)^{2}(2 m \lambda+1)^{2}}\times\\
&\left[4-\frac{\left[\omega_{2}(1-\beta)+9 \omega_{3}-4 \omega_{4}\right]^{2}}{\omega_{4}\left[\omega_{1}(1-\beta)^{2}-2 \omega_{2}(1-\beta)-12 \omega_{3}+4 \omega_{4}\right]}\right].
\end{split}
\end{equation}

\vspace{1mm}
Next suppose  if $\beta \in [0, \tau]$ and the  maximum of $K(\rho)$ occurs at $\rho=2$. Then 

\begin{equation*}
    \begin{aligned}
        \max \{K(\rho): \rho \in [0, 2]\}&=K(2)=\frac{4(1-\beta)^{2}}{(\gamma+1)^{2}(m \lambda+1)}\left[\frac{(m+1)^{2}(1-\beta)^{2}}{(\gamma+1)^{2}(m \lambda+1)^{3}}\right.\\&\qquad\qquad\qquad\qquad\qquad\qquad\left.+\frac{6}{(\gamma+2)(\gamma+3)(3 m \lambda+1)}\right].
    \end{aligned}
\end{equation*}
We only now need to note (see the idea in the  second part of Result 2) if $\beta \in (\tau, 1)$ then the  maximum of $K(\rho)$ cannot occur at $\rho=2 $ since
$$
K(2) \leq \frac{4(1-\beta)^{2}}{(\gamma+1)^{2}(\gamma+2)^{2}(2 m \lambda+1)^{2}}\times
\left[4-\frac{\left[\omega_{2}(1-\beta)+9 \omega_{3}-4 \omega_{4}\right]^{2}}{\omega_{4}\left[\omega_{1}(1-\beta)^{2}-2 \omega_{2}(1-\beta)-12 \omega_{3}+4 \omega_{4}\right]}\right]  \,(=K(\rho_2)).
 $$

\vspace{1mm}
Finally let us consider $\beta \in [0,1)$ and  the  maximum of $K(\rho)$ occurring at $\rho=0$. Then 
$$
 \max \{K(\rho): \rho \in [0, 2]\}=K(0)=\frac{16(1-\beta)^{2}}{(\gamma+1)^{2}(\gamma+2)^{2}(2 m \lambda+1)^{2}}.
$$
We note  (see the ideas in the  second part of Result 2) if $\beta \in (\tau, 1)$ then the  maximum of $K(\rho)$ cannot occur at $\rho=0 $ since
$$
K(0) \leq \frac{4(1-\beta)^{2}}{(\gamma+1)^{2}(\gamma+2)^{2}(2 m \lambda+1)^{2}}\times
\left[4-\frac{\left[\omega_{2}(1-\beta)+9 \omega_{3}-4 \omega_{4}\right]^{2}}{\omega_{4}\left[\omega_{1}(1-\beta)^{2}-2 \omega_{2}(1-\beta)-12 \omega_{3}+4 \omega_{4}\right]}\right]  \,(=K(\rho_2)).
$$
Finally note (see the ideas in Result 1 and the details in Result 2) if
$$
\beta \in\left[0,1-\frac{\omega_{2}+\sqrt{\omega_{2}^{2}+\omega_{1}\left[12 \omega_{3}-4 \omega_{4}\right]}}{\omega_{1}}\right],
$$
or

$$
\beta \in\left(1-\frac{\omega_{2}+\sqrt{\omega_{2}^{2}+\omega_{1}\left[12 \omega_{3}-4 \omega_{4}\right]}}{\omega_{1}}, 1\right),
$$
then the  maximum of $K(\rho)$ cannot occur at $\rho=0$ since $K(0) \leq K(2)$.

This completes the proof.
\end{proof}

By setting $\lambda=1$ and $\gamma=0$ in Theorem 2.1, we obtain the following consequence.

\begin{Corollary}\cite{A24}
 Let $f \in \Xi_{\Sigma_{\mathrm{m}}}(\beta)\quad(0 \leq \beta<1)$ be given by (1.5). Then

$$
\left|a_{m+1} a_{3 m+1}-a_{2 m+1}^{2}\right| \leq\left\{\begin{array}{cc}
\frac{4(1-\beta)^{2}}{m+1}\left[\frac{(1-\beta)^{2}}{m+1}+\frac{1}{3 m+1}\right], \qquad\qquad\qquad\quad\qquad\qquad \beta \in[0, v] \\
\\
\frac{(1-\beta)^{2}}{(2 m+1)^{2}}\left[4-\frac{\left[m(1-\beta) \psi_{1}+3 \psi_{2}-2 \psi_{3}\right]^{2}}{\psi_{3}\left[(2 m+1)(1-\beta)^{2} \psi_{1}-m(1-\beta) \psi_{1}+\psi_{3}-2 \psi_{2}\right]}\right], \beta \in[v, 1)
\end{array}\right.
$$
where

$$
\begin{aligned}
& \psi_{1}:=(2 m+1)(3 m+1), \\
& \psi_{2}:=(m+1)(2 m+1)^{2}, \\
& \psi_{3}:=(m+1)^{2}(3 m+1),
\end{aligned}
$$
and

$$
v:=\frac{(3 m+1)(7 m+4)-\sqrt{m^{2}(3 m+1)^{2}+8 \psi_{2}(3 m+1)}}{4 \psi_{1}}.
$$
\end{Corollary}

By taking $m=1$ in Theorem 2.1, we conclude the following result.

\begin{Corollary}
    Let $f \in \Xi_{\Sigma}(\lambda, \gamma ; \beta)\quad\left(\lambda \geq 1, \gamma \in \mathbb{N}_{0}, 0 \leq \beta<1\right)$ be given by (1.1). Then

$$
\left|a_{2} a_{4}-a_{3}^{2}\right| \leq\left\{\begin{array}{c}
\frac{8(1-\beta)^{2}}{(\gamma+1)^{2}(\lambda+1)}\left[\frac{2(1-\beta)^{2}}{(\gamma+1)^{2}(\lambda+1)^{3}}+\frac{3}{(\gamma+2)(\gamma+3)(3 \lambda+1)}\right], \qquad\qquad\quad \beta \in[0, \xi] \\
\\
\frac{4(1-\beta)^{2}}{(\gamma+1)^{2}(\gamma+2)^{2}(2 \lambda+1)^{2}}\left[4-\frac{\left[\vartheta_{2}(1-\beta)+9 \vartheta_{3}-4 \vartheta_{4}\right]^{2}}{\vartheta_{4}\left[4 \vartheta_{1}(1-\beta)^{2}-2 \vartheta_{2}(1-\beta)-12 \vartheta_{3}+4 \vartheta_{4}\right]}\right], \beta \in[\xi, 1)
\end{array}\right.
$$
where

\begin{eqnarray*}
&&\vartheta_{1}:=(\gamma+2)^{2}(\gamma+3)(2 \lambda+1)^{2}(3 \lambda+1), \\&&
\vartheta_{2}:=(\gamma+1)(\gamma+2)(\gamma+3)(\lambda+1)^{2}(2 \lambda+1)(3 \lambda+1), \\&&
\vartheta_{3}:=(\gamma+1)^{2}(\gamma+2)(\lambda+1)^{3}(2 \lambda+1)^{2}, \\&&
\vartheta_{4}:=(\gamma+1)^{2}(\gamma+3)(\lambda+1)^{4}(3 \lambda+1),
\end{eqnarray*}
and
 
$
\quad\qquad\qquad\quad\qquad\qquad\quad\xi:=1-\frac{\vartheta_2+\sqrt{\vartheta_2^2+48 \vartheta_1 \vartheta_3}}{8 \vartheta_1}.
$
\end{Corollary}

\begin{Remark}
 Corollary 2.2 improves a result in Altinkaya and Yalçin \cite[Theorem 3]{A26}.
\end{Remark}

By putting $\gamma=0$ in Corollary 2.2, we obtain the following result.

\begin{Corollary}
    Let $f \in \Xi_{\Sigma}(\lambda ; \beta)\quad(\lambda \geq 1,0 \leq \beta<1)$ be given by (1.1). Then

$$
\left|a_2 a_4-a_3^2\right| \leq\left\{\begin{array}{cc}
\frac{8(1-\beta)^2}{\lambda+1}\left[\frac{2(1-\beta)^2}{(\lambda+1)^3}+\frac{1}{2(3 \lambda+1)}\right], \qquad\qquad \qquad\beta \in[0, \epsilon] \\
\\
\frac{2(1-\beta)^2}{(2 \lambda+1)^2}\left[4-\frac{\left[\eta_2(1-\beta)+3 \eta_3-2 \eta_4\right]^2}{\eta_4\left[8 \eta_1(1-\beta)^2-2 \eta_2(1-\beta)-4 \eta_3+2 \eta_4\right]}\right], \beta \in[\epsilon, 1)
\end{array}\right.
$$    
where

\begin{eqnarray*}
&&\eta_1:=(2 \lambda+1)^2(3 \lambda+1), \\&&
\eta_2:=(\lambda+1)^2(2 \lambda+1)(3 \lambda+1), \\&&
\eta_3:=(\lambda+1)^3(2 \lambda+1)^2, \\&&
\eta_4:=(\lambda+1)^4(3 \lambda+1),
\end{eqnarray*}
and

$$
\epsilon:=1-\frac{(\lambda+1)^2(3 \lambda+1)+\sqrt{(\lambda+1)^4(3 \lambda+1)^2+32(\lambda+1)^3(2 \lambda+1)^2(3 \lambda+1)}}{16(2 \lambda+1)(3 \lambda+1)}.
$$
\end{Corollary}

\begin{Remark}
    Corollary 2.3 improves a result in Altinkaya and Yalçin \cite[Corollary 5]{A26}.
\end{Remark}

By setting $\lambda=1$ in Corollary 2.3 , we get the following consequence.

\begin{Corollary}\cite{A27}
    Let $f \in \Xi_{\Sigma}(\beta)\quad(0 \leq \beta<1)$ be given by (1.1). Then

$$
\left|a_2 a_4-a_3^2\right| \leq\left\{\begin{array}{ll}
(1-\beta)^2\left[(1-\beta)^2+\frac{1}{2}\right], & \beta \in\left[0, \frac{11-\sqrt{37}}{12}\right] \\
\\
\frac{(1-\beta)^2}{16}\left[\frac{60 \beta^2-84 \beta-25}{9 \beta^2-15 \beta+1}\right], & \beta \in\left[\frac{11-\sqrt{37}}{12}, 1\right)
\end{array} .\right.
$$
\end{Corollary}

\begin{Remark}
    Corollary 2.4  recovers  a result in Altinkaya and Yalçin \cite[Corollary 4]{A26}.
\end{Remark}

\section{Concluding remarks}\label{sec:con}
In this investigation, we consider a constructed subclass $\Xi_{\Sigma_m}(\lambda, \gamma ; \beta)$ of the class $\Sigma_m$ of $m$-fold symmetric bi-univalent functions and several properties of the results are discussed. Moreover, with a  specialization of the parameters, some consequences of the class are mentioned and they improve some existing upper bounds for $\mathrm{H}_{2}(2)$  on certain subclasses of 1-fold symmetric bi-univalent functions.

\section*{Data Availability}
Data sharing not applicable to this article as no datasets were generated or analyzed during the current study.

\section*{Funding}
Not applicable.

\section*{Author's contributions}
Conceptualization, P.O.S.; Data curation, S.J.M.F.; Formal analysis, N.C.; Funding acquisition, N.C.; Investigation, P.O.S., R.P.A., P.O.M., N.C. and T.A.; Methodology, P.O.S, R.P.A., N.C.and T.A.; Project administration, R.P.A.; Software, P.O.S., S.J.M.F. and P.O.M.; Supervision, T.A.; Validation, R.P.A. and P.O.M.; Visualization, P.O.S.; Writing – original draft, P.O.S., R.P.A. and P.O.M.; Writing – review \& editing, P.O.S. and R.P.A. All of the authors read and approved the final manuscript.

\section*{Acknowledgements}
Researchers Supporting Project number (RSP2023R153), King Saud University, Riyadh, Saudi Arabia.

\section*{Declarations}

\noindent {\bf Ethical approval}\; Not applicable.

\noindent {\bf Competing interests}\; The authors declare no competing interests.


\end{document}